\documentclass[a4paper]{amsart}

\usepackage{geometry}
\usepackage{setspace}
\usepackage[utf8]{inputenc}
\usepackage[T1]{fontenc}
\usepackage[english]{babel}
\usepackage{textcmds}

\usepackage{amssymb,amsmath,amsfonts,amsthm}
\usepackage[all, cmtip]{xy}
\usepackage{tikz}   
\usetikzlibrary{cd}
\usepackage{longtable}

\usepackage{algorithm}
\usepackage{algpseudocode}

\usepackage{hyperref}
\usepackage{url}
\usepackage{verbatim}

\usepackage{color}

\usepackage[alphabetic,lite]{amsrefs} 

\newtheorem{theorem}{Theorem}[section]
\newtheorem{definition}[theorem]{Definition}
\newtheorem{proposition}[theorem]{Proposition}

\theoremstyle{remark}
\newtheorem{example}[theorem]{Example}

\newcommand{\ZZ}{\mathbb{Z}}

\newcommand{\RR}{\mathbb{R}}
\newcommand{\CC}{\mathbb{C}}

\newcommand{\res}{\operatorname{res}}

\DeclareMathOperator{\GL}{GL}
\DeclareMathOperator{\PGL}{PGL}

\DeclareMathOperator{\PGmL}{P\Gamma L}
\DeclareMathOperator{\GmL}{\Gamma L}
\DeclareMathOperator{\chr}{char}
\DeclareMathOperator{\Aut}{Aut}

\DeclareMathOperator{\Hom}{Hom}
\DeclareMathOperator{\Gal}{Gal}

\DeclareMathOperator{\id}{id}
\DeclareMathOperator{\conj}{conj}

\DeclareMathOperator{\tra}{tra}

\begin{document}

\title[Semi-Projective Representations and Twisted Representation Groups]{Semi-Projective Representations and Twisted Representation Groups}

\author{Massimiliano Alessandro, Christian Gleissner and Julia Kotonski}

\thanks{
\textit{Acknowledgements:} The authors would like to thank Andreas Demleitner and Himanshu Shukla for useful comments and discussions. Moreover, they thank an anonymous referee for a plethora of helpful suggestions to improve the exposition and the proof of Proposition~\ref{Stem}. The first author was partially supported by PRIN 2020KKWT53 003 - Progetto \emph{Curves, Ricci flat Varieties and their Interactions}.}
 
\address{Massimiliano Alessandro \newline  Universit\`a degli Studi di Genova, DIMA Dipartimento di Matematica, I-16146 Genova, Italy\newline University of Bayreuth, Universit\"atsstr. 30, D-95447 Bayreuth, Germany}
\email{alessandro@dima.unige.it, massimiliano.alessandro@uni-bayreuth.de}

\address{Christian Gleissner \newline University of Bayreuth, Universit\"atsstr. 30, D-95447 Bayreuth, Germany}
\email{christian.gleissner@uni-bayreuth.de}

\address{Julia Kotonski \newline University of Bayreuth, Universit\"atsstr. 30, D-95447 Bayreuth, Germany}
\email{julia.kotonski@uni-bayreuth.de}

\begin{abstract}
    A semi-projective representation is a homomorphism of a finite group into the group of semi-projective transformations of a finite dimensional vector space over a field. Schur's concept of a representation group for projective representations is extended to semi-projective representations under the assumption that the field is algebraically closed. A computer algorithm is given that produces, for a given finite group, all such \emph{twisted} representation groups under trivial or conjugation actions on the field of complex numbers.
\end{abstract}

\maketitle

\section{Introduction}
Throughout this paper $K$, $V$ and $G$ will denote a field, a non-trivial finite dimensional $K$-vector space and a finite group respectively.
In  \cite{S04},
Schur developed the theory of projective representations, which are homomorphisms from a group $G$ to the group of projective transformations 
$\PGL(V)$. We consider semi-projective representations of $G$, these are homomorphisms from $G$ to the group of semi-projective transformations $\PGmL(V)$. 
Here,  $\PGmL(V)$ is defined as the quotient of the group of semi-linearities $$\GmL(V)\cong \GL(V)\rtimes \Aut(K)$$  modulo the action of the multiplicative group $K^{\ast}$. 
A  semi-projective representation involves  an action $\varphi$ of $G$ on $K$ by  automorphisms.  In this way, $K^{\ast}$ becomes a $G$-module and one can consider
  the second cohomology group $H_{\varphi}^2(G,K^{\ast})$ with respect to the action. 
In analogy to the projective case, this group plays an important role as it is the  obstruction space of  the lifting  problem of semi-projective representations to semi-linear representations, i.e., homomorphisms 
from $G$ to $\GmL(V)$.

As our main result, we show  that if $K$ is algebraically closed, then  for any given action $\varphi$ of $G$, there exists a finite $\varphi$-twisted representation group $\Gamma$, which has the property that  any semi-projective representation 
inducing the action $\varphi$ admits  a semi-linear lift to $\Gamma$. Despite the fact that $\Gamma$ is in general  not unique, it has minimal order among all groups enjoying the lifting  property. This allows us to study semi-projective representations of $G$ via semi-linear representations of $\Gamma$. 
We also  give a cohomological characterization of a group $\Gamma$ to be a $\varphi$-twisted representation group, which reduces to the classical description of a representation group in the case that  the action $\varphi$ is trivial. 
In general, it seems to be difficult to determine explicitly a $\varphi$-twisted representation group, even when $\varphi$ is trivial. 
We approach this problem in the semi-projective case via an algorithm for the case 
$K=\CC$ under the assumption that $\varphi$ takes values in $\Gal(\CC/\RR)$; this produces all $\varphi$-twisted representation groups of a given group $G$.\\
There are situations where semi-projective representations arise naturally, for example in \emph{Clifford theory}:  
Isaacs  \cite{I81} developed the concept 
of crossed-projective representations,  which is analogous to our notion of semi-projective representations, in order  to study the problem of extending 
 $G$-invariant irreducible $K$-representations of a normal subgroup of $G$ to  $G$ for $K$ arbitrary.\\
 A second situation occurred in a recent work of Demleitner and the second author (see \cite{DG22}), where they constructed certain quotients of complex tori by free holomorphic actions of finite groups  and  investigated
their homeomorphism and biholomorphism classes. Homeomorphisms and biholomorphisms of such quotients are induced by affine transformations. When determining  the linear parts of these transformations, a homomorphism from a finite group to $\PGL(n,\CC)\rtimes\Gal(\CC/\RR)$ arose. Moreover, they  had to 
determine a particular kind of lift of this map to $\GL(n,\CC)\rtimes \Gal(\CC/\RR)$. This example served for us as a motivation to extend Schur's theory to semi-projective representations.\\

In Section~\ref{S2}, we introduce semi-projective representations and discuss their interplay with group cohomology. We  state the lifting problem and give a cohomological criterion 
 for a semi-projective representation of $G$  to lift to a semi-linear representation of an extension of 
 $G$ by a finite abelian group $A$. 
 In Section \ref{S3}, we construct a 
 $\varphi$-twisted representation group for any given $G$ together with an action $\varphi$ on an algebraically closed field $K$. For this purpose, we adapt the construction of a representation group in the projective case (see \cite[Chapter~11]{I94}) to our setup. Then we give a cohomological characterization of a $\varphi$-twisted representation group 
  and show that it coincides with the classical notion in the case that $\varphi$ is the trivial action. Section \ref{ExSection} is devoted to basic examples of semi-projective representations and twisted representation groups. Moreover, an algorithm is developed which allows to determine all $\varphi$-twisted representation groups of a given $G$ under the assumption that  $K=\CC$ and $\varphi$ maps to  $\Gal(\CC/\RR)$. Running a MAGMA implementation, we determine the $\varphi$-twisted representation  groups of  $\mathcal D_4$ for all possible actions $\varphi$.

\section{Preliminaries}\label{S2}

 In this section, we introduce  semi-linear and semi-projective  representations and discuss some of their basic properties. 

\begin{definition}
A bijective map $f\colon V \to V$ is called a \emph{semi-linear transformation} if there exists an automorphism $\varphi_f \in \Aut(K)$ such that 
\[
 f(v+w)=f(v)+f(w) \qquad \makebox{and} \qquad 
f(\lambda v)= \varphi_f(\lambda) f(v)
\]
 for all $v, w \in V$ and all $\lambda \in K$.
The set of all semi-linear transformations of $V$ forms a group $\GmL(V)$. 
\end{definition}

The group $\GmL(V)$  contains $\GL(V)$ as a normal subgroup and sits inside the following short exact sequence 
\[
1 \longrightarrow \GL(V) \longrightarrow \GmL(V) \longrightarrow \Aut(K) \longrightarrow 1. 
\]
This sequence  splits, so that $\GmL(V) \cong   \GL(V) \rtimes \Aut(K)$.\\
Let $v_1, \ldots, v_n$ be a basis of $V$. Then we can associate to every $f \in \GmL(V) $ an invertible  matrix $A_f:=(a_{ij})$ by  
\[
f(v_j)= \sum_{i=1}^n a_{ij}v_i. 
\]
This procedure establishes an isomorphism between $\GmL(V)$ and the semidirect product \makebox{$\GL(n,K) \rtimes \Aut(K)$} with group operation 
\[
(A, \varphi) \cdot (B, \psi) := (A \varphi(B), \varphi \circ \psi). 
\]
Here, $\varphi(B)$ is the matrix obtained by applying  the automorphism $\varphi$ to the  entries of $B$.\\

In analogy to  the group of \emph{projective transformations} $\PGL(V)$,  the group of \emph{semi-projective transformations} $\PGmL(V)$ is defined  as the quotient of 
 $\GmL(V)$ modulo the equivalence relation 
\[
		f \sim g \quad  \makebox{if and only if there exists  $\lambda \in K^*$, such that}  \quad    f=\lambda g. 
\]
By construction, the sequence 
\[
1 \longrightarrow K^{\ast} \longrightarrow \GmL(V) \longrightarrow \PGmL(V) \longrightarrow 1
\]
is exact. The structure of $\PGmL(V)$ is similar to the one of $\GmL(V)$: the group $\PGL(V)$ is a normal subgroup of $\PGmL(V)$, and there is a split  exact sequence 
\[
1 \longrightarrow \PGL(V) \longrightarrow \PGmL(V) \longrightarrow \Aut(K) \longrightarrow 1. 
\]
Note  that the map $\PGmL(V) \to \Aut(K)$ is well-defined  because the automorphism attached to an element in $\PGmL(V)$ is independent  of the representative. 
After choosing a \emph{projective frame}, $\PGmL(V)$ can be identified with the semidirect product  
\[
\PGL(n,K) \rtimes \Aut(K).
\]

 We now introduce our main objects of study: 

\begin{definition}
A \emph{semi-linear representation}  is a homomorphism 
$F \colon G \to \GmL(V)$ and a \emph{semi-projective representation}  is a homomorphism 
$f \colon G \to \PGmL(V)$.  
\end{definition}

Isaacs \cite{I81} used the term \emph{crossed-projective representation} for a semi-projective representation under the identification of $\PGL(n,K)\rtimes\Aut(K)$ with $\PGmL(V)$.\\

The lifting problem: Every semi-linear representation $F \colon G \to \GmL(V)$ induces a semi-projective representation $f \colon G \to \PGmL(V)$  by composition with the quotient map:
	\[
	\begin{tikzcd}
		\GmL(V)\arrow{r} & \PGmL(V)\\
		G \arrow{u}{F}\ar[ur, "f"']
	\end{tikzcd}
	\]
However, it is not true that every semi-projective representation can be obtained in this way. 
The obstruction to  the existence of a lift to  $\GmL(V)$, or more generally, the interplay between 
semi-linear and semi-projective representations,   can be  described by using  \emph{group cohomology} in analogy  to the 
classical theory of projective representations.\\
Given a semi-linear or semi-projective representation of $G$, we obtain an action 
\[
\varphi\colon G \to \Aut(K), \quad  g \mapsto \varphi_g, 
\]
by composition with the projection from $\GmL(V)$ or $\PGmL(V)$ to $\Aut(K)$, respectively. 
Via this  action, the abelian 
group $K^{\ast}$ obtains  a $G$-module structure;
 in particular, we can define the group of cocycles $Z_\varphi^i(G,K^{\ast})$, coboundaries $B_\varphi^{i}(G,K^{\ast})$ and  the cohomology groups 
\[
H_\varphi^i(G,K^{\ast})= Z_\varphi^i(G,K^{\ast})/B_\varphi^{i}(G,K^{\ast}).
\]
Here, the subscript $\varphi$ is used to emphasize that the $G$-module structure of $K^\ast$ is in general not trivial, and might be dropped  in the trivial case.
For details on group cohomology, see \cite{B94}.\\
The basic observation is that we can associate to every  semi-projective representation a well-defined class in the second cohomology group:  

\begin{proposition}
Let $f\colon G \to  \PGmL(V)$ be a semi-projective representation and let $f_g$ be a representative of the class $f(g)$ for each $g \in G$. Then there exists a map $\alpha \colon G \times G \to K^{\ast}$ such that $f_{gh}= \alpha(g,h) (f_g \circ  f_h)$ for all $g,h \in G$. 
The map $\alpha$  is a $2$-cocycle, i.e., 
\[
\alpha(g,hk)\varphi_g\big(\alpha(h,k)\big)=\alpha(g,h)\alpha(gh,k).  
\]
The cohomology class $[\alpha] \in H_\varphi^2(G,K^{\ast})$ is independent of the chosen representatives $f_g$. 
\end{proposition}
\begin{proof}
Since $f$ is a homomorphism, $[f_{gh}]=[f_g] \circ [f_h]$, which implies that $f_{gh}=\alpha(g,h)(f_g \circ f_h)$  for $\alpha(g,h)\in K^{\ast}$.  Now 
\begin{align*}
f_{g(hk)}= \alpha(g,hk)  (f_g \circ  f_{hk}) &= \alpha(g,hk)  (f_g \circ  \alpha(h,k)  (f_h \circ  f_k)) \\
& = \alpha(g,hk)  \varphi_g(\alpha(h,k))  ( f_g \circ f_h \circ  f_k).
\end{align*}  
On the other hand,
\[
f_{(gh)k}= \alpha(gh,k)  (f_{gh} \circ  f_{k})= \alpha(gh,k)  \alpha(g,h)  (f_g \circ  f_h \circ  f_k). 
\]
Comparing the two expressions yields 
\[
\alpha(g,hk)  \varphi_g(\alpha(h,k))=\alpha(gh,k)  \alpha(g,h).
\]
Let $f'_g$ be another representative for $f(g)$, then there exists $\tau(g) \in K^{\ast}$ such that 
$f_g=\tau(g)f'_g$. Let $\alpha'$ be defined by $f'_{gh}= \alpha'(g,h)  (f'_g \circ  f'_h)$ for all $g,h \in G$.
A  computation as above shows that 
\[
\alpha'(g,h)=\varphi_g(\tau(h))  \tau(gh)^{-1}  \tau(g)  \alpha(g,h),
\]
where
$\partial \tau(g,h) =\varphi_g(\tau(h))  \tau(gh)^{-1}  \tau(g)$ is a $2$-coboundary.
\end{proof}
Let  $f\colon G \to  \PGmL(V)$ be a semi-projective representation.
     Choosing $\id_V$ as a representative for $f(1)$, then the 
     2-cocycle $\alpha$ is \emph{normalized}, that is, $\alpha(1,g)=\alpha(g,1)=1$ for all $g\in G$. If $f$ is induced by a semi-linear representation $F$, then the attached cohomology class is trivial. Conversely, assume that $\alpha$ is a coboundary, that is $     \alpha(g,h)=
     \varphi_g(\tau(h))  \tau(gh)^{-1}  \tau(g)$ for some function $\tau \colon G \to K^{\ast}$. Then the map 
     \[
    F\colon G \to \GmL(V), \quad g \mapsto F_g:=\tau(g)f_g
     \]
     is a semi-linear representation inducing $f$, as the following computation shows: 
     \begin{align*}
F_g \circ F_h = (\tau(g)  f_g)\circ (\tau(h)  f_h)&=  \tau(g)  \varphi_g(\tau(h))  (f_g\circ f_h) \\
& = \tau(gh)  \alpha(g,h)    (f_g \circ f_h) \\
& =\tau(gh)  f_{gh} =F_{gh}. 
\end{align*} 
We have assigned  to every  semi-projective representation an element in $H_\varphi^2(G, K^{\ast})$. In fact, all cohomology classes arise in this way:
let  $\varphi  \colon G \to  \Aut(K)$ be an action of $G$ of order $n$ on $K$ and $\alpha\in Z_\varphi^2(G, K^{\ast})$.  In analogy to the \emph{regular representation}, 
consider  the vector space $V$ with basis  $\lbrace e_{h} ~\vert ~h\in G \rbrace$ and define 
for every $g\in G$ an element $R_g \in \GL(V)$ via $R_g(e_h):= \alpha(g,h)^{-1}e_{gh}$. Then the map 
\[
f \colon G \to \PGL(V) \rtimes \Aut(K), \quad g \mapsto ([R_g],\varphi_g),
\]
is a semi-projective representation with assigned  cohomology class $[\alpha] \in H_\varphi^2(G, K^{\ast})$. 

We have thus shown that if 
$H_\varphi^2(G, K^{\ast})$ is non-trivial, then there are semi-projective representations without a semi-linear lift.
In the projective case, this problem  was first noticed and investigated by Schur.
In order to study projective representations by means of  ordinary linear representations, 
he and subsequent authors, under certain conditions on $K$, constructed a representation group $\Gamma$ of $G$: in modern terminology, a \emph{stem extension} 
\[
1 \longrightarrow A \longrightarrow \Gamma \longrightarrow G \longrightarrow 1, \qquad \makebox{with} \:  A \cong H^2(G,K^{\ast}).
\]
\emph{Stem} means that $A$ is central and contained in the commutator subgroup $[\Gamma,\Gamma]$.
Such an extension has the property that for every projective representation  $f \colon G \to \PGL(V)$ (see \cite[Chapter~3.3]{K85} for details), there exists an ordinary linear representation $F \colon \Gamma \to \GL(V)$ fitting into the following commutative diagram
	\[
	\begin{tikzcd}
		1\arrow{r} & A \arrow{r}\arrow{d} & \Gamma \arrow{r}\arrow{d}{F} & G\arrow{r}\arrow{d}{f} & 1\\
		1\arrow{r} & K^\ast \arrow{r} & \GL(V)\arrow{r} & \PGL(V)\arrow{r} & 1
	\end{tikzcd}
	\]
In this scenario, we say that \emph{$F$ induces $f$} or that \emph{$f$ can be lifted to $F$} and use similar terminology in the semi-projective case below.\\

Here are some facts about group extensions: let $1\to A \to \Gamma \to G \to 1$ be an extension of $G$ by a finite abelian group $A$ and $s\colon G\to \Gamma$ a set-theoretic section. There is an  action of $G$ on $A$ defined by 
$g\ast a:=s(g) a  s(g)^{-1}$. 
Since $A$ is abelian, the action is independent of the choice of the section. Now, $s(gh)= \beta(g,h)s(g)s(h)$ for some $\beta(g,h) \in A$ leading to a map $\beta\colon G \times G \to A$, which is a  2-cocycle since  
\[
(g \ast\beta(h,k)) \beta(g,hk)=\beta(gh,k)\beta(g,h). 
\]
A different choice of section $s'\colon G \to \Gamma$ yields a cohomologous cocycle $\beta'$. Hence, we can associate to the given extension a unique cohomology class in $H_\varphi^2(G,A)$. Suppose now that an action $\varphi  \colon G \to  \Aut(K)$ on the field $K$ is given. Then, by composition with the projection $\Gamma \to G$, we also obtain an action of $\Gamma$ on $K$ with kernel containing $A$.
In this situation, the \emph{inflation-restriction exact sequence} of Hochschild and Serre \cite[Theorem~2]{HoSe53} reads:
\[
    \begin{tikzcd}
        1 \arrow{r} & H_\varphi^1(G,K^{\ast})\arrow{r}{\inf} & H_\varphi^1(\Gamma,K^{\ast})\arrow{r}{\res} & \Hom_G(A,K^{\ast})\arrow{r}{\tra} & H_\varphi^2(G,K^{\ast})\arrow{r}{\inf} & H_\varphi^2(\Gamma,K^{\ast})
    \end{tikzcd}
\]
Here, \emph{inf} and \emph{res} are induced by inflation and restriction of cocycles and the \emph{transgression} map \emph{tra} is defined as 
\[
\tra \colon \Hom_G(A,K^{\ast}) \to H_\varphi^2(G,K^{\ast}), \quad \lambda \mapsto [\lambda \circ \beta].
\]
Clearly,  this map depends only on the cohomology class of $\beta$. Using this terminology, we obtain:

\begin{theorem}\label{Lifting}
Let $1\to A \to \Gamma \overset{\pi}{\to} G \to 1$ be an extension of $G$ by a finite abelian group $A$ with associated cohomology class 
$[\beta] \in H_\varphi^2(G,A)$.  
A semi-projective representation  $f \colon G \to P\Gamma L(V)$  with class  $[\alpha] \in H_\varphi^2(G,K^{\ast})$ 
 is induced by a semi-linear representation 
 \[
 F \colon \Gamma \to \Gamma L(V), \quad \gamma \mapsto F_{\gamma},
\] 
 if and only if 
$[\alpha]$ belongs to the image of the transgression map. \end{theorem}

\begin{proof}
Assume that $f$ is induced by a semi-linear representation $F$. By assumption, there exists a function $\lambda \colon \Gamma \to K^{\ast}$ such that $F_\gamma=\lambda(\gamma)f_{\pi(\gamma)}$ for all $\gamma \in \Gamma$. Since we may assume that $f_1=\id$, it follows that $F_a =\lambda(a) f_{\pi(a)} =\lambda(a) \id$ for all $a \in A$. As a result, the restriction $\lambda_A$ is a homomorphism. 
We claim that $\lambda \in \Hom_G(A,K^{\ast})$, that is $\lambda(g\ast a)= \varphi_g(\lambda(a))$ for all $g \in G$ and $a \in A$, and say that $\lambda$ is \emph{$G$-equivariant}. Now,
\[
\varphi_g(\lambda(a)) \id = F_{s(g)} \circ (\lambda(a) \id) \circ F_{s(g)^{-1}}= F_{s(g)} \circ F_a \circ F_{s(g)^{-1}}= F_{s(g)as(g)^{-1}}= \lambda(g\ast a) \id. 
\]
By using the definition of $\beta$, we compute
\begin{align*}
 F_{s(gh)}=F_{\beta(g,h)s(g)s(h)}  &= F_{\beta(g,h)} \circ F_{s(g)}\circ F_{s(h)}  \\
& =   \lambda(\beta(g,h))  (\lambda(s(g))f_{g}) \circ (\lambda(s(h))f_{h}) \\
& =  \lambda(\beta(g,h))  \lambda(s(g))   \varphi_g(\lambda(s(h))  (f_g \circ f_h).
\end{align*} 
On the other hand,
\begin{align*}
F_{s(gh)} & = \lambda(s(gh))f_{gh} =  \lambda(s(gh))  \alpha(g,h)  (f_g \circ  f_h).
\end{align*}  
Comparing the results, we obtain $\alpha(g,h)= \lambda(\beta(g,h))  \partial(\lambda \circ s)(g,h)$, hence
\[
[\lambda \circ \beta] = [\alpha] \in H_{\varphi}^2(G, K^{\ast}).
\]
Conversely, assume there is a function $\tau \colon G \to K^{\ast}$ and $\lambda \in \Hom_G(A,K^{\ast})$ such that 
\[
\alpha(g,h)= \lambda(\beta(g,h))  \varphi_g(\tau(h))    \tau(gh)^{-1}   \tau(g).
\]
Define
\[
F \colon \Gamma \to \Gamma L(V), \quad a  s(g) \mapsto \lambda(a) \tau(g)f_g,
\]
then one can  show that $F$ is a homomorphism inducing $f$.
\end{proof}

 A natural question arises: is it possible to find for every $G$ together with a given action $\varphi \colon G \to \Aut(K)$ 
an extension 
\[
1 \longrightarrow A \longrightarrow \Gamma \longrightarrow G \longrightarrow 1 \qquad \makebox{with $A$ finite and abelian} 
\]
such that every semi-projective representation $f \colon G \to \PGmL(V)$ with action $\varphi$ is induced by a semi-linear representation $F \colon \Gamma \to  \GmL(V)$?

From Theorem \ref{Lifting}, answering this question amounts to constructing an extension with surjective trangression map
\[
\tra \colon \Hom_G(A,K^{\ast}) \to H_\varphi^2(G,K^{\ast}), \quad \lambda \mapsto [\lambda \circ \beta].
\]
Clearly, this is possible only if $H_\varphi^2(G,K^{\ast})$ is finite. 
In case such an extension $\Gamma$ exists, its order is bounded from below: 
\[
 \vert G \vert  \vert H_\varphi^2(G,K^{\ast})\vert \leq  \vert G \vert  
\vert \Hom(A,K^{\ast}) \vert \leq 
 \vert G \vert  \vert A \vert = \vert \Gamma \vert. 
\]

Unfortunately, $H_\varphi^2(G,K^{\ast})$ is in general not finite. Nevertheless, in  many important situations $H_\varphi^2(G,K^{\ast})$ is  finite: for example, if $K$ is algebraically closed and $\varphi \colon G \to \Aut(K)$ is an arbitrary action, see \cite[Theorem~11.15]{I94} and note that the proof carries over to non-trivial actions and all cohomology groups $H_\varphi^i(G,K^\ast)$, where $i\geq 1$.

\section{Twisted Representation Groups: the Algebraically Closed Case}\label{S3}
 
Throughout this section, we will assume $K$ is algebraically closed and there is a fixed action $\varphi \colon G \to \Aut(K)$. Under these assumptions, we are mainly dealing with a case similar to  $K=\CC$, where  $\varphi$ acts just by the identity and/or complex conjugation. Indeed, $H:=\varphi(G)$ is a finite group and   $F:=K^H \subset K$ is a Galois extension with Galois group $H$. The Artin-Schreier Theorem \cite{AS27} implies that if $H$ is non-trivial, then it is isomorphic to  $\ZZ_2$, $K=F(i)$  with $i^2=-1$ and $\chr(K)=0$.
In particular, if $\chr(K) \neq 0$, then the action is necessarily trivial and we are in the projective setting.\\
The main result of this section is the following. 
\begin{theorem} \label{ConstructionRepGroup}
There exists an extension 
\[
1\longrightarrow A \longrightarrow \Gamma \longrightarrow G \longrightarrow 1
\]
of $G$ with $A$ finite and abelian such that the transgression map 
$\tra \colon \Hom_G(A,K^{\ast}) \to H_\varphi^2(G,K^{\ast})$ is an isomorphism. 
\end{theorem} 
	
\begin{proof}
    From \cite[Lemma~11.14]{I94}, we may let $M'$ be a finite complement of $B_\varphi^2(G,K^*)$ in $Z_\varphi^2(G,K^*)$. For each $m'\in M'$, define $m:=(\partial \tau_{m'})m'$, where $\tau_{m'}(g)=\varphi_g(m'(1,1)^{-1})$. Then the set of these $m$ forms a complement $M$ of normalized $2$-cocycles. Define an action on $A:= \Hom(M,K^*)$ via $\varphi$ by $(g * a) (m):= \varphi_g(a(m))$ for all $m\in M$.
    Now define $\beta \colon  G \times G \to A $ by $\beta(g,h)(m):=m(g,h)$ for all $m \in M$.
    A straightforward computation confirms that $\beta\in Z^2(G,A)$ under the action of $G$ on $A$ and clearly $\beta$ is normalized.
    Then it is clear (see \cite[Chapter~IV]{B94})  that $\beta$ defines an extension 
\[
1 \longrightarrow A \longrightarrow \Gamma\longrightarrow G \longrightarrow 1,
\]
where $\Gamma := A \times  G$ with  binary operation $(a,g) \cdot (b, h):= (a (g *b) \beta(g,h), gh)$.
Note that the conjugation action of $G$ on $A$ is given by $g * a$. Now, we claim that  the transgression map
\begin{align*}
    \tra\colon \Hom_G(A, K^*) \to H_\varphi^2(G, K^*),\quad    \lambda \mapsto [\lambda \circ \beta],
\end{align*}
is surjective.
Any class in $H_\varphi^2(G,K^*)$ is  represented by a (unique) element $m_0 \in M \leq Z_\varphi^2(G,K^*)$. 
Consider the evaluation homomorphism at $m_0$, that is 
	\[
   		\lambda \colon A \to K^{\ast}, \quad a \mapsto a(m_0).
	\]
Note that 	$\lambda$ is $G$-equivariant, in fact
  	\[
    	\lambda(g \ast a)=(g\ast a)(m_0) = \varphi_g(a(m_0))=\varphi_g(\lambda(a)) \qquad \makebox{for all}\: g\in G.
    \]
Furthermore, 
 	\[
    	(\lambda\circ \beta)(g,h)=\lambda(\beta(g,h))=\beta(g,h)(m_0)=m_0(g,h)
    \]
and thus $\tra$ is surjective. Finally, $\tra$ is injective because
\begin{equation*} 
    \vert M\vert = \vert H_\varphi^2(G,K^{\ast}) \vert \leq \vert \Hom_G(A,K^{\ast})  \vert \leq \vert \Hom(A,K^{\ast})  \vert  \leq \vert A \vert \leq \vert M \vert. \qedhere
\end{equation*}
\end{proof}

From the above chain of inequalities, it follows that 

    \begin{enumerate}
        \item all characters of $A$ are $G$-equivariant, namely   $\Hom_G(A,K^{\ast}) = \Hom(A,K^{\ast}) $,
        \item $A \cong \Hom(A,K^\ast)$, 
        \item $A \cong H_\varphi^2(G,K^\ast)$,
        \item the  group $\Gamma$ has minimal order $\vert \Gamma \vert= \vert G \vert \vert H_\varphi^2(G,K^{\ast}) \vert $,
        \item 
        $H_\varphi^1(G,K^{\ast}) \cong H_\varphi^1(\Gamma,K^{\ast})$ using the inflation-restriction sequence
        \[
        0 \longrightarrow H_\varphi^1(G,K^{\ast}) \longrightarrow H_\varphi^1(\Gamma,K^{\ast})  \longrightarrow \Hom_G(A,K^{\ast}) \overset{\sim}{\longrightarrow} H_\varphi^2(G,K^{\ast}).
        \]
    \end{enumerate}
    Note that (5) is equivalent to $\tra$ being injective.\\
If $\chr(K)\neq 0$, the action $\varphi$ is trivial. Moreover, property (2) is equivalent to $\chr(K) \nmid \vert A \vert $, which from (3) confirms the known result that  $\chr(K) \nmid \vert H^2(G, K^\ast) \vert$ (see \cite[Theorem~2.3.2]{K85}).\\
 The above observations motivate the following definition:

\begin{definition}\label{RepGroup}
Let $\varphi  \colon G \to  \Aut(K)$  be an action of a finite group $G$ on an algebraically closed field $K$. A 
group $\Gamma$ is called a \emph{$\varphi$-twisted representation group} of $G$ if  there exists an extension 
\[
1 \longrightarrow A \longrightarrow \Gamma \longrightarrow G \longrightarrow 1 \qquad \makebox{with $A$ finite and abelian}
\]
such that the following conditions hold:

\begin{enumerate}
    \item $\chr(K) \nmid \vert A \vert $, 
    \item $\Hom_G(A,K^{\ast})=\Hom(A,K^{\ast})$,
    \item the transgression map
    \[
    \tra \colon \Hom_G(A,K^{\ast}) \to H_\varphi^2(G,K^{\ast})
    \]
    is an isomorphim.
\end{enumerate}

\end{definition}

Next a numerical criterion is given to decide whether an extension is a $\varphi$-twisted representation group.

\begin{proposition}\label{criterion}
 Let 
\[
1 \longrightarrow A \longrightarrow \Gamma \longrightarrow G \longrightarrow 1
\]
be an extension by a finite  abelian group $A$. Then $\Gamma$ is a $\varphi$-twisted representation group of $G$ if and only if the following conditions are satisfied:
\begin{enumerate}
\item
$\vert A \vert= \vert H_\varphi^2(G,K^{\ast}) \vert $, 
\item
$|\Hom_G(A,K^{\ast})|=|\Hom(A,K^{\ast})|$ and
\item
$\vert H_\varphi^1(G,K^{\ast})\vert  =\vert  H_\varphi^1(\Gamma,K^{\ast}) \vert $. 
\end{enumerate}

\end{proposition}

\begin{proof}
Clearly, every $\varphi$-twisted representation group fulfills the three conditions. Conversely, if they hold, then the  inflation-restriction sequence together with (3) implies that $H_\varphi^1(G,K^{\ast})\cong H_\varphi^1(\Gamma,K^{\ast})$ and hence  the transgression map is injective. Condition~(1) implies $\chr(K)\nmid \lvert A\rvert$. Therefore, by using condition (2), we have 
\[
\Hom_G(A,K^{\ast}) =  \Hom(A,K^{\ast}) \cong A. 
\]
Thus, the transgression map  is also surjective. 
\end{proof}

\begin{proposition} \label{Stem}
If $\varphi  \colon G \to  \Aut(K)$ is the trivial action, then an extension as in Definition \ref{RepGroup} is a stem extension.  
\end{proposition}

\begin{proof}
Since $\varphi$ is trivial, the inflation-restriction sequence is 
\[
1 \longrightarrow \Hom(G,K^{\ast}) \longrightarrow \Hom(\Gamma,K^{\ast})  \longrightarrow \Hom(A,K^{\ast}) \longrightarrow H^2(G,K^{\ast}).
\]
As the transgression map is an isomorphism,  restriction  $\Hom(\Gamma,K^{\ast})  \to \Hom(A,K^{\ast})$ is trivial. Let $H_\Gamma:=\Hom(\Gamma^{ab},K^\ast)$ and $H_A:=\Hom(A,K^\ast)$. Then $\cap_{\bar{\lambda} \in H_\Gamma}\ker(\bar{\lambda})$ is the Sylow $p$-subgroup of $\Gamma^{ab}$ for $p:=\chr(K) >0$ and is trivial for $p=0$. Now, $\lambda_{A}$ is trivial for each lift of $\bar{\lambda}$ to $\Gamma$ and hence $A\leq [\Gamma,\Gamma]$ since $p\nmid \lvert A\rvert$. Similarly, for $a\in A$, $\lambda(s(g)as(g)^{-1})=\lambda(a)$ for all $\lambda\in H_A$ and all $g\in G$. Thus $s(g)as(g)^{-1}a^{-1}\in \cap_{\lambda\in H_A}\ker(\lambda)$, which is trivial. Hence $A\leq Z(\Gamma)$.
\end{proof}

 The next proposition shows that Definition~\ref{RepGroup} is an extension of the definition of a representation group for $\varphi$ trivial.
\begin{proposition}
When the $G$-action  on $K$ is trivial,  Definition \ref{RepGroup} is equivalent  to the classical definition of a representation group (see \cite[Corollary 11.20]{I94}):
\begin{enumerate}
    \item the extension 
        $    1 \longrightarrow A \longrightarrow \Gamma \longrightarrow G \longrightarrow 1
        $
        is stem,
    \item
    $|A|=|H^2(G,K^\ast)|$.
\end{enumerate}

\end{proposition}

\begin{proof}
 Suppose the extension satisfies the conditions of Definition~\ref{RepGroup}, then Proposition \ref{Stem} implies that it is  stem and (2) is obviously true. Conversely, suppose we have a stem extension 
 \[
    1 \longrightarrow A \longrightarrow \Gamma \longrightarrow G \longrightarrow 1
 \]
 such that $|A|=|H^2(G,K^\ast)|$. Then
 $\chr(K) \nmid |H^2(G,K^\ast)|$ as previously noted. Next $\Hom_G(A,K^{\ast})=\Hom(A,K^{\ast})$ since $A\leq Z(\Gamma)$. Furthermore, the inflation-restriction sequence yields that the transgression map is injective since $A\leq [\Gamma,\Gamma]$, and hence $\Hom(G,K^\ast)\cong \Hom(\Gamma,K^\ast)$.  Finally, $\Hom(A,K^\ast)\cong A$ since $\chr(K)\nmid\lvert A\rvert$ and so the transgression map is also surjective by (2).
\end{proof}

Note that the order of a $\varphi$-twisted representation group $\Gamma$ is unique, whereas the group itself is in general not (see examples in Section \ref{ExSection}), even in the projective case.

\section{Examples}\label{ExSection}

\begin{example}
Consider  $K=\CC$  as a $G=\ZZ_2$-module, where $1\in \ZZ_2$ acts  via  complex conjugation $\varphi(1)(z)=\conj(z)= \overline{z}$. 
In this example,  a twisted representation group  $\Gamma$ is of order $4$ because 
\[
H_\varphi^2(\ZZ_2,\CC^{\ast}) \cong 
(\CC^{\ast})^{\ZZ_2}/  N_{\CC/\RR}(\CC^{\ast}) \cong \RR^\ast/\RR^+\cong  \ZZ_2.
\] 
The transgression map is required to be  an isomorphism, so the 
  extension 
 \[
 0 \longrightarrow \ZZ_2 \longrightarrow \Gamma \longrightarrow \ZZ_2 \longrightarrow 0 
 \]
 has to be non-split, which  implies $\Gamma \cong \ZZ_4$.  
Consider the semi-projective representation 
\[ 
f \colon \ZZ_2 \to \PGL(2,\CC)\rtimes\ZZ_2, \quad 1 \mapsto \bigg( \bigg[\begin{pmatrix} 0 & -1 \\ 1 & 0 \end{pmatrix}\bigg], \conj\bigg).
\]
Its  cohomology class in $H_{\varphi}^2(\ZZ_2, \CC^*)$ is represented by the  normalised $2$-cocycle $\alpha$  with  
$\alpha(1,1)=-1$. 
It has no lift to a semi-linear representation of $\ZZ_2$ but a semi-linear  lift to $\Gamma$ is given by 
\[ 
F \colon \ZZ_4 \to \GL(2,\CC)\rtimes\ZZ_2, \quad 1 \mapsto \bigg( \begin{pmatrix} 0 & -1 \\ 1 & 0 \end{pmatrix}, \conj \bigg).
\]
\end{example}

In the following, we explain  how to use a computer algebra system, such as  MAGMA \cite{BCP97}, to produce all  twisted representation groups of a given finite group $G$  in the case $K=\CC$. We assume that $\varphi\colon G\to \Aut(\CC)$ takes values in $\Gal(\CC/\RR)\cong \lbrace \id, \conj\rbrace$. \\
The conditions in Proposition~\ref{criterion} need to be satisfied. Since we want to use  a computer, it is necessary to replace 
the module $\CC^{\ast}$ by a discrete module.
Identifying complex conjugation with multiplication by $-1$, the homomorphism $\varphi$ induces an action of $G$ on $\ZZ$ that is  also  denoted  by $\varphi$. 
In this way, we can consider   $\varphi$  as a complex character of $G$ of degree $1$ with values in  $\{\pm1\}$. Furthermore,
the exponential sequence
\[
\xymatrixcolsep{2pc}
\xymatrix{
0 \ar[r] & \ZZ \ar[r]^{\cdot 2\pi i } & \CC \ar[r]^{\exp} & \CC^\ast \ar[r] & 1
}
\]
 becomes a sequence of $G$-modules. Since the cohomology groups $H_\varphi^n(G,\CC)$ vanish for $n\geq 1$, see \cite[Corollary III.10.2]{B94},  
the corresponding long exact  sequence induces isomorphisms
\[
    H_\varphi^n(G,\CC^\ast)\cong H_\varphi^{n+1}(G,\ZZ) \qquad \makebox{for all} \quad   n\geq 1.
\]
Similarly, these isomorphisms hold for the cohomology groups of $\Gamma$.\\
These considerations lead to the MAGMA implementation on the webpage 
\begin{center}
\url{http://www.staff.uni-bayreuth.de/~bt300503/publi.html}.
\end{center}

 It takes as inputs $G$ and $\varphi$, which is given as a character with values in $\lbrace \pm 1 \rbrace$, and it returns all $\varphi$-twisted representation groups of $G$.

\begin{example}
Running  our code, we compute the $\varphi$-twisted representation groups of the \emph{dihedral group}
\[
\mathcal D_4=\langle s,t \mid s^2=t^4=1,\:sts^{-1}=t^3\rangle
\]
for all possible actions $\varphi\colon \mathcal D_4\to\Aut(\CC)$:\\
\begin{center}
\begin{tabular}{c|c|c|c}
     $\varphi(s)$ & $\varphi(t)$ & $A=H_\varphi^2(\mathcal D_4,\CC^*)$ & $\varphi$-twisted representation groups\\
     \hline
    1 & 1 & $\ZZ_2$  & $\langle 16,7\rangle, \: \langle 16,8\rangle,\: \langle 16,9\rangle$\\
    -1 & -1 &$\ZZ_2\times\ZZ_2$ & $\langle 32, 14\rangle,\: \langle 32, 13\rangle$\\
    1  & -1 &$\ZZ_2\times\ZZ_2$ & $\langle 32, 9\rangle,\: \langle 32, 10\rangle,\: \langle 32, 14\rangle,\: \langle 32, 13\rangle$\\
    -1 & 1 &$\ZZ_2\times\ZZ_2$& $\langle 32, 2\rangle,\: \langle 32, 10\rangle ,\: \langle 32, 13\rangle$
\end{tabular}
\end{center}

Here, the symbol $\langle n,d\rangle$ denotes  the $d$-th group of order $n$ in MAGMA's \emph{Database of Small Groups}.

\end{example}

 \end{document}